\documentclass[12pt]{article}
\usepackage{amsthm,amssymb,amsmath}
\usepackage{epsfig}

\title{Planar digraphs of digirth four are 2-colourable}

\author{Zhentao Li\\
  {D\'epartement d'Informatique, UMR CNRS 8548}\\
  {\'Ecole Normale Sup\'erieure}\\
  {75005 Paris}\\
  email: {\tt zhentao.li@ens.fr}
\and
Bojan Mohar\thanks{Supported in part by an NSERC Discovery Grant (Canada),
  by the Canada Research Chair program, and by the
  Research Grant P1--0297 of ARRS (Slovenia).}~\thanks{On leave from:
  IMFM, Ljubljana, Slovenia. The results of this paper were obtained while the author visited \'Ecole Normale Sup\'erieure in Paris. The hospitality of the hosting university is greatly acknowledged.}\\
  {Department of Mathematics}\\
  {Simon Fraser University}\\
  {Burnaby, B.C. V5A 1S6} \\
  email: {\tt mohar@sfu.ca}
}

\newtheorem{theorem}{Theorem}[section]
\newtheorem{lemma}[theorem]{Lemma}
\newtheorem{corollary}[theorem]{Corollary}
\newtheorem{conjecture}[theorem]{Conjecture}

\newcommand{\DEF}[1]{{\emph{#1\/}}}

\begin{document}

\maketitle

\begin{abstract}
Neumann-Lara conjectured in 1985 that every
planar digraph with digirth at least three is 2-colourable, meaning
that the vertices can be 2-coloured without creating any monochromatic directed
cycles. We prove a relaxed version of this conjecture: every planar digraph
of digirth at least four is 2-colourable.
\end{abstract}


\section{Introduction}

Let $D$ be an oriented graph (i.e., a digraph without cycles of length at most $2$). A function $f: V(D) \to \{1,\dots,k \}$ is a \DEF{$k$-colouring} of $D$ if the subdigraph induced by vertices of colour $i$ is acyclic for all $i$.
We say that $D$ is \DEF{$k$-colourable} if it admits a $k$-colouring.

The following conjecture was proposed by Neumann-Lara \cite{N} (and independently by \v{S}krekovski, see~\cite{BFJKM2004}).

\begin{conjecture}
\label{conj:planar}
Every oriented planar graph is $2$-colourable.
\end{conjecture}

There seems to be a lack of methods to attack Conjecture \ref{conj:planar}, and only sporadic partial results are known.

The \DEF{digirth} of a digraph is the length of its shortest directed
cycle. Harutyunyan and one of the authors \cite{HM16} proved Conjecture \ref{conj:planar} under additional assumption that the digirth of $D$ is at least five. Their proof is based on elaborate use of the (nowadays standard) discharging
technique. However, it is unlikely that the same method can be pushed further when directed 4-cycles are allowed.

The main result of this note is the following theorem whose proof is based on a novel technique, at least when colourings of graphs are concerned.

\begin{theorem}
\label{thm:main}
Every oriented planar graph with a vertex $v_0$ such that each directed cycle of length $3$ uses $v_0$ is $2$-colourable.
\end{theorem}

\begin{corollary}
\label{cor:digirth4}
Every planar digraph with digirth at least\/ $4$ is $2$-colourable.
\end{corollary}

The rest of the paper is devoted to the proof of Theorem \ref{thm:main}.

\section{Proof of Theorem \ref{thm:main}}

The main tool that we will use in our proof is the notion of a \emph{Tutte path}. This is a special kind of a path that was first used by Tutte \cite{Tu56,Tu77} in his proof that every 4-connected planar graph is hamiltonian. A version used in this paper is taken from \cite{Th}.

Tutte paths are defined using the following notion of connectivity. If $G$ is a graph and $H$ is a subgraph of $G$, then an \emph{$H$-component} $B$ of $G$ is either an edge $e\in E(G)\setminus E(H)$ with both ends in $H$ or it is a connected component $Q$ of $G-V(H)$ together with all edges from $Q$ to $H$ and all ends of these edges. The vertices of $V(B) \cap V(H)$ are called the \emph{vertices of attachment} of $B$. A bridge with $k$ vertices of attachment is said to be \emph{$k$-attached}.

Let $G$ be a graph with cycle $C$ and let $u,v$ be two vertices in $G$. A path $P$ in $G$ from $v$ to $u$ is called a \emph{Tutte path} with respect to $C$ if
\begin{itemize}
\item[(i)]
each $P$-component has at most three vertices of attachment and
\item[(ii)]
each $P$-component containing an edge of $C$ has at most two vertices of attachment.
\end{itemize}

\begin{theorem}[Thomassen \cite{Th}]
\label{thm:Tutte path}
Let $G$ be a 2-connected plane graph with facial cycle $C$. Let $v$ and $e$
be a vertex and edge, respectively, of $C$ and let $u$ by any vertex of $G$ distinct from $v$. Then $G$ has a Tutte path with respect to $C$ from $u$ to $v$ that contains the edge $e$.
\end{theorem}

\begin{lemma}
\label{lem:combining a colouring for clique sum}
Let $D$ and $D'$ be digraphs, whose intersection is a tournament $T$. Then any colouring of $D$ and any colouring of $D'$ that agree on $V(T)$ form a colouring of $D\cup D'$.
\end{lemma}

\begin{proof}
The only detail to verify is that the combined colouring does not produce a monochromatic cycle. Since the colourings of $D$ and $D'$ have no monochromatic cycles, such a directed cycle $C$ would be composed of $2r$ ($r\ge1$) directed paths $P_1\cup P_2 \cup \cdots \cup P_{2r}$, where each $P_{i}$ is a directed path completely contained in either $D$ or $D'$ (joining vertices $v_i$ and $v_{i+1}$ in $T$ ($i=1,\dots,r$, indices taken modulo $r$). Since none of these paths together with the arcs of the tournament on $T$ forms a directed cycle, $T$ contains the arcs $v_iv_{i+1}$, which all together form a monochromatic directed cycle in $T$ and hence in both $D$ and $D'$. This contradiction completes the proof.
\end{proof}

\begin{lemma}
\label{lem:outer triangle colouring extend}
Let $G$ be a triangulation of the plane with the outer face bounded by a triangle $T=abc$. Then for any orientation $D$ of $G$ with digirth at least $4$ and any precolouring of\/ $T$ with colours $1$ and $2$, there exists a $2$-colouring of $D$ that extends the precolouring on $T$.
\end{lemma}

The first case of the proof of this lemma is similar to that of Wu for induced forests \cite{Wu}.

\begin{proof}
The proof is by induction, with the base of induction corresponding to the case when the triangulation is 4-connected. We consider the dual graph $H$ of $G$. Note that $H$ is a cubic graph and that it is cyclically 4-edge-connected, i.e., any 3-edge-cut in $H$ isolates a single vertex from the rest of the graph. We distinguish two cases, 
depending on whether the precolouring on $T$ uses just one or both colours (see Figure~\ref{fig:4connected precolouring}).

%\begin{figure}
%\epsfig{file=4connected.eps}
%\caption{The two 4-connected cases. The dual edges are dashed. The cycle formed by the Tutte path and the edge between its endpoints is in green. The inside is coloured 1 (red) and outside 2 (blue).}
%\label{fig:4connected precolouring}
%\end{figure}

\begin{figure}[htb]
   \centering
   \includegraphics[width=12.5cm]{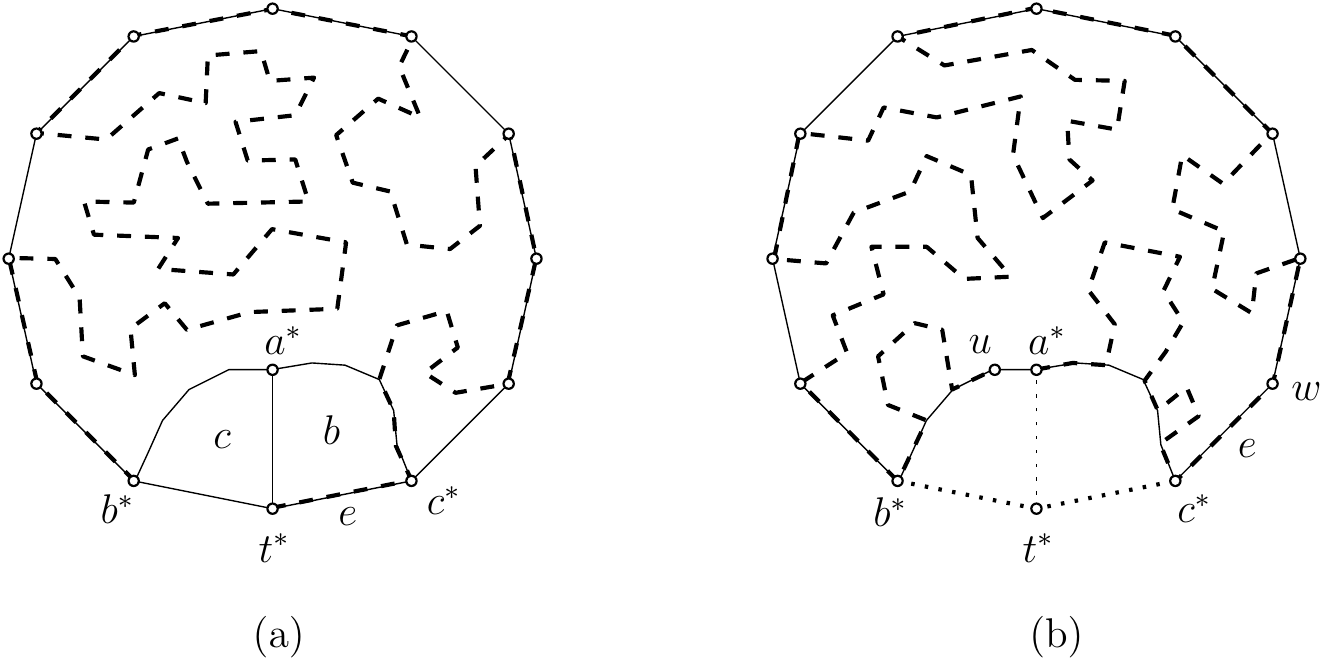}
   \caption{Finding Tutte path $P^*$ in the 4-connected case. (a) Bicoloured triangle $abc$. The dual graph $H$ is shown, where $a,b,c$ are faces. (b) Monochromatic triangle $abc$.}
   \label{fig:4connected precolouring}
\end{figure}

\paragraph{4-connected -- two colours case.} Suppose first that not all vertices of $T$ are precoloured the same. We may assume that $a$ is coloured 1 and that $b,c$ are coloured 2. Let us denote the vertices of $H$ corresponding to $T$ and to the faces surrounding $T$ by $t^*,a^*,b^*,c^*$, where the face $a^*$ of $G$ contains vertices $b,c$ but not $a$, etc. Let $P^*$ be a Tutte path in $H$ with respect to the facial cycle of $H$ containing $b^*,t^*,c^*$ (i.e., the face dual to the vertex $a$) that connects $t^*$ and $b^*$ and passes through the edge $e=t^*c^*$. This path, whose existence follows from Theorem \ref{thm:Tutte path}, together with the edge $b^*t^*$ forms a cycle $C^*$ in $H$. The cycle crosses the edges $ac$ and $ab$ of $T$. We may assume that the exterior of $C^*$ contains $b$ and $c$, while $a$ is in its interior. Now we colour the vertices of $D$ in the interior of the cycle by using colour 1, and those in the exterior by 2.

We claim that the above gives a 2-colouring of $D$. To see this, assume that the process gives rise to a monochromatic directed cycle $R$. Then we may assume that the whole cycle $R$ lies in the interior of $C^*$. Since this is also a cycle in $G$, its interior contains a vertex of $H$, but since $C^*$ is a Tutte path, any such cycle separates a $C^*$-component with at most 3 attachments. Since $G$ is 4-connected, this corresponds to a vertex in $H$ and the cycle $R$ is a triangle. But $D$ has digirth at least 4, so $R$ is not a directed cycle. This contradiction proves the claim.

\paragraph{4-connected -- one colour case.} Suppose now that all three vertices of $T$ have the same colour. In this case we proceed in the similar way except that we consider the graph $H$ obtained from the dual of $G$ by deleting the vertex $t^*$. The resulting graph is 2-connected and vertices $a^*,b^*,c^*$ all lie on the outer facial cycle $C^*$. Let $u$ be a neighbor of $a^*$ on $C^*$ and $w$ a neighbor of $c^*$ in $C^*$. Now we take a Tutte path with respect to $C^*$ from $a^*$ to $u$ that passes through the edge $e=c^*w$. Together with the edge $ua^*$ we obtain a cycle $R$, and we colour the vertices of $D$ inside this cycle differently from the colour used on $a,b,c$; and the vertices outside this cycle the same as $a,b,c$.

The above gives an extension of the colouring of $T$, as desired. We only need to argue that there are no monochromatic directed cycles of $D$. As before, the cycles of $G$ that are monochromatic are triangles that are dual to vertices in 3-attached $R$-components. All of these correspond to facial triangles in $G$ that are not directed cycles in $D$ by the assumption on the digirth. The only possible difference from the first case might be an $R$-component containing the vertex $b^*$ (when $b^*\notin V(R)$). By property (ii) of Tutte cycles, this $R$-component in $H$ is 2-attached. But in the dual graph of $G$, it is adjacent to $t^*$ and thus gives a monochromatic subgraph of $D$ that contains $a,b,c$ and the vertex $b'\ne b$ forming the second facial triangle with the edge $ac$ of $G$. However, since all edges incident with $a^*$ and $c^*$ are in $R$, this subgraph only consists of four vertices. The two triangles in this subgraph are not directed cycles in $D$, which implies that also the 4-cycle in this subgraph cannot be a directed cycle. This shows that $D$ has no monochromatic directed cycles.

\paragraph{Non-4-connected case.} If $G$ is not 4-connected, then there is a triangle $T'=xyz$ that separates the graph, one component being the subgraph $D_0$ in the exterior of $T'$ (including $T'$, which now becomes a facial triangle) and the other one, $D_1$, formed from $T'$ and the vertices and edges in the interior of $T'$. Now we first extend the precolouring of $T$ to $D_0$ (by using the induction hypothesis), and then apply the induction hypothesis to $D_1$ to extend the colouring of $T'$ obtained in the first step. By Lemma \ref{lem:combining a colouring for clique sum}, the combined colouring is a 2-colouring of $D$.
\end{proof}

\begin{proof}[Proof of Theorem \ref{thm:main}]
We may assume the underlying undirected graph $G$ is triangulated as otherwise, we can add a vertex inside each face of $D$ adjacent to all vertices of that face in $G$ and direct all edges towards the new vertices in $D$. This creates no new directed cycles.

The proof is now essentially the same as the proof of Lemma \ref{lem:outer triangle colouring extend} with one difference the we will explain below. If $G$ has a separating triangle, the induction step is the same as in Lemma \ref{lem:outer triangle colouring extend} (by applying either the lemma or this theorem inductively). So, it suffices to consider the 4-connected case, and we are not bound with any precolouring.

The faces of $G$ containing $v_0$ form a cycle $C^*$ in the dual graph $G^*$. Let us choose three consecutive vertices $u^*,v^*,w^*$ on $C^*$. A Tutte path with respect to $C^*$ joining $u^*$ and $v^*$ and containing the edge $v^*w^*$ forms a cycle $R$ together with the edge $u^*v^*$. By the connectivity of $G^*$, each $R$-component with two attachments is an edge. So all edges of $C^*$ are either an edge of $C^*$ or an entire $P$-component by property (ii) of Tutte paths.
In particular, no vertex of $C^*$ is in a 3-attached $R$-component and hence their duals, triangles containing $v_0$, all have a vertex inside $R$ and a vertex outside $R$. Thus, using colour 1 for all vertices inside $R$ and using colour 2 outside $R$ gives an acyclic colouring of $D$.
\end{proof}

\end{document}